\documentclass[11pt]{article}
 \usepackage{amsmath,amssymb,amsthm,amsfonts}
 \usepackage{epstopdf}
 \RequirePackage[dvips]{graphicx}
 \usepackage[dvips]{epsfig}
 \usepackage{color}
\usepackage[usenames,dvipsnames]{pstricks}
\usepackage{tikz}
 \def\draw #1 by #2 (#3){
  \vbox to #2{
    \hrule width #1 height 0pt depth 0pt
    \vfill
    \special{picture #3} % this is the low-level interface
    }
  }

 \def\scaleddraw #1 by #2 (#3 scaled #4){{
  \dimen0=#1 \dimen1=#2
  \divide\dimen0 by 1000 \multiply\dimen0 by #4
  \divide\dimen1 by 1000 \multiply\dimen1 by #4
  \draw \dimen0 by \dimen1 (#3 scaled #4)}
  }

\newtheorem{theorem}{Theorem}[section]
\newtheorem{example}[theorem]{Example}
\newtheorem{problem}[theorem]{Problem}

\newtheorem{defin}[theorem]{Definition}

\newtheorem{nt}{Note}

 \newtheorem{rule-def}[theorem]{Rule}

\begin{document}
%\baselineskip 16pt
 %\setcounter{chapter}{1}
 % defining short form------
 \newcommand{\la}{\lambda}
 \newcommand{\si}{\sigma}
 \newcommand{\ol}{1-\lambda}
 \newcommand{\be}{\begin{equation}}
 \newcommand{\ee}{\end{equation}}
 \newcommand{\bea}{\begin{eqnarray}}
 \newcommand{\eea}{\end{eqnarray}}

\author{Kittitat Iamthong\thanks{Department of Mathematics and Statistics, University of Strathclyde, 26 Richmond Street, Glasgow,
  G1 1XH, United Kingdom}, Ji-Hwan Jung\thanks{Center for Educational Research, Seoul National University, Seoul 08826, Republic of Korea}, Sergey Kitaev\footnotemark[1] \\
{{\footnotesize
kittitat.iamthong@strath.ac.uk},\,{\footnotesize jihwanjung@snu.ac.kr},\,{\footnotesize
sergey.kitaev@strath.ac.uk}}}

\title{{Encoding labelled $p$-Riordan graphs by words and pattern-avoiding permutations}
\date{}}

\maketitle

 \begin{abstract} The notion of a $p$-Riordan graph generalizes that of a Riordan graph, which, in turn, generalizes the notions of a Pascal graph and a Toeplitz graph. In this paper we introduce the notion of a $p$-Riordan word, and show how to encode $p$-Riordan graphs by $p$-Riordan words. For special important cases of Riordan graphs (the case $p=2$) and oriented Riordan graphs (the case $p=3$) we provide alternative encodings in terms of pattern-avoiding permutations and certain balanced words, respectively. As a bi-product of our studies, we provide an alternative proof of a known enumerative result on closed walks in the cube. 

 \bigskip
 
 
 \noindent{\bf Mathematics Subject Classification}:
%Primary:
05A05, 05A15\\
%secondary:.

 \end{abstract}

 \section{Introduction}\label{intro}
The interplay between words and graphs has been investigated in the literature repeatedly, as discussed in \cite{KitLoz} focusing on {\em word-representable graphs} that are also surveyed in \cite{Kit2017}. Various ways to encode graphs in terms of words often allow to reveal and describe various useful properties of graphs, such as classes where difficult algorithmic problems become easy. The {\em Pr\"{u}fer sequence} giving a fascinating relationship between labelled trees with $n$ vertices and sequences of length $n-2$ made of the elements of the set $[n]:=\{1,2,\ldots,n\}$, is a classical example showing the importance of words for graph enumeration.
  
In this paper we study encoding of so-called labelled $p$-Riordan graphs by certain words, and also consider encoding a particular important case of $p=2$ (labelled Riordan graphs) by certain pattern-avoiding permutations extensively studied in the literature \cite{Kit2011}. We consider only labelled graphs, so the word ``labelled'' is omitted throughout this paper. 

We begin with introducing the notions of interest.\\

\noindent
{\bf Riordan graphs.} A {\em
Riordan matrix} \cite{shap} $L=[\ell_{ij}]_{i,j\ge0}$ generated by two formal
power series $g=\sum_{n=0}^\infty g_nt^n$ and $f=\sum_{n=1}^\infty
f_nt^n$ in ${\mathbb Z}[[t]]$ is denoted as $(g,f)$ and defined as
an infinite lower triangular matrix whose $j$-th column generating
function is $gf^j$, {\it i.e.} $\ell_{ij}=[t^i]gf^j$ where
$[t^k]\sum_{n\ge0}a_nt^n=a_k$. If $g_0\ne0$ and $f_1\ne
 0$ then the Riordan matrix is called {\it proper}.

A simple graph $G$ of order $n$ is said to be a {\em Riordan graph}
if the adjacency matrix $A(G)$ can be expressed as
\begin{equation} A(G) \equiv (tg,f)_n + (tg,f)_n^T \pmod{2} \label{riordandef}\end{equation}
for some generating functions $g$ and $f$ over $\mathbb Z$ where
$(tg,f)_n$ is the $n\times n$ leading principle matrix of the
Riordan matrix $(tg,f)$. A Riordan graph $G$ on $n$ vertices with
the adjacency matrix $A(G)$ given by (\ref{riordandef}) is denoted
as $G=G_n(g,f)$. If we let $A(G)=(a_{i,j})_{1\le i,j\le n}$, then, by (\ref{riordandef}),
for $i\ge j$,
\begin{align*}
a_{i,j}=a_{j,i}\equiv [t^{i-2}]gf^{j-1} \pmod{2}.
\end{align*}
In particular, if $[t^0]g \equiv [t^1]f \equiv 1 \pmod{2}$, then the graph $G_n(g,f)$ is called {\it proper}. We denote the set of Riordan graphs with $n$ vertices by $\mathcal{RG}_n$. When defining Riordan graphs we assume, without loss the generality, that $g$ and $f$ have binary coefficients. The number $r_n$ of Riordan graphs of order $n\ge1$ is known from  \cite{CJKM} to be 
\begin{eqnarray}\label{Riord-rn}
r_n={4^{n-1}+2\over 3}.
\end{eqnarray}

Riordan graphs were introduced in \cite{CJKM,CJKM2} and they are a far-reaching
generalization of the well-known and well studied Pascal graphs and
Toeplitz graphs, and also some other families of graphs. The Riordan
graphs are proved to have a number of interesting (fractal)
properties  \cite{CJKM}, and  spectral properties of Riordan graphs were studied in~\cite{CJKM2}.

There are several naturally defined classes/families of Riordan
graphs~\cite{CJKM,CJKM2}. A Riordan graph $G_n(g,t)$ is said to be of the 
{\em Appell type}. Riordan graphs of the Appell type are also
known as {\em Toeplitz graphs}.
%Especially, the Toeplitz graph $G_n\left({1\over 1-z-z^2},z\right)$ is called the {\em Fibonacci graph} and is denoted by $FG_n$.
Toeplitz graphs have been studied in \cite{CJKM,DTTVZZ,NP}.  A Riordan graph $G_n(g,tg)$ is said to be of the {\em Bell type} and
it has been studied in \cite{CJKM,J1}. The well-known {\em Pascal graph} $G_n\left(\frac{1}{1-t},\frac{t}{1-t}\right)$ is a particular example of a Bell type Riordan graph, and it is the only such graph of type $(1+f,f)$, also considered in our paper. Finally, a Riordan graph $G_n(f',f)$ is said to be of the {\em derivative type} and it has been studied in \cite{CJKM}. \\

\noindent
{\bf Oriented Riordan graphs and $p$-Riordan graphs.} There is a natural generalization of the notion of a Riordan graph that was introduced in \cite{J}. This notion is obtained by replacing ``mod 2'' by ``mod $p$'' in the definition of a Riordan graph. While the definition makes sense for any integer $p\geq 2$, it is normally assumed that $p$ is a prime number to resolve the invertibility issues preventing us from being able to analyze such graphs. In particular, the number $r^{(p)}_n$ of $p$-Riordan graphs for a prime $p$  was derived in \cite{J}, and it is given by  
\begin{eqnarray}\label{mod-p-formula}r^{(p)}_n=\frac{p^{2(n-1)}+p}{p+1},\end{eqnarray}
while no enumeration is known for a non-prime $p$. Setting $p=2$ in (\ref{mod-p-formula}) we recover the number of Riordan graphs given by (\ref{Riord-rn}), while setting $p=3$ in (\ref{mod-p-formula}) we obtain the number $\tilde{r}_n$ of oriented Riordan graphs of order $n$: 
\begin{eqnarray}\label{formula-mod-3}\tilde{r}_n=\frac{3^{2(n-1)}+3}{4}.\end{eqnarray} 
For $p>3$, $p$-Riordan graphs can be thought of as weighted Riordan graphs. 

In the case of a $p$-Riordan graph $G$, in some contexts it is convenient to let the elements of the adjacency matrix $A(G)$ be coming from the set $\{\lfloor
p/2\rfloor,\ldots,-1,0,1,\ldots,\lfloor p/2\rfloor\}$. However, in this paper it is more convenient to let these elements be in the set $\{0,1,\ldots,p-1\}$. In particular, oriented Riordan graphs have adjacency matrices' elements in $\{0,1,2\}$. We denote the set of $p$-Riordan graphs with $n$ vertices by $\mathcal{RG}_{n}^{(p)}$. \\

\noindent
{\bf Patterns in permutations.} We think of permutations to be written in the one-line notation, and a permutation of length $n$ is called an $n$-permutation. The reduce form of a permutation $\pi$ is the permutation red$(\pi)$ obtained from $\pi$ by substituting the $i$-th smallest element by $i$. For example, red(4287)=2143.

A {\em pattern} is a permutation. We say that a pattern $\tau=\tau_1\cdots\tau_k$ {\em occurs} in a permutation $\pi=\pi_1\cdots\pi_n$ if there are indices $1\leq i_1<i_2<\cdots<i_k\leq n$ such that red$(\pi_{i_1}\pi_{i_2}\cdots\pi_{i_k})=\tau$. For example, the permutation $14235$ has four occurrence of the pattern 123, namely, the subsequences 145, 125, 135 and 235, while the permutation 34512 {\em avoids} the pattern 132 (has no occurrences of it). Permutation patterns have been the subject of extensive research in the literature \cite{Kit2011}.

A well-known result \cite{Kit2011} states that the number of $n$-permutations that avoid simultaneously the patterns 123 and 132 is $2^{n-1}$. This can be proved using a bijection $\psi$ from the set of binary words of length $n-1$ to the set of restricted $n$-permutations that is  denoted by $S_n(123,132)$. Think of constructing a permutation in $S_n(123,132)$ by inserting elements $1, 2,\ldots, n$ one by one, starting from 1 and continuing with the least available element, into $n$ empty slots and observing that at each step an element $i$ can only be inserted in one of the two rightmost  empty slots to avoid the patterns 123 and 132; the element $n$ will be inserted in a unique way. Thus, given a binary word $b_1b_2\cdots b_{n-1}$, $\psi$ uses the process of insertion, and  places the element  $i$ to the left (resp., right) available slot if $b_i=0$ (resp., 1). For example, $\psi(011001)=7546312$. \\

\noindent
{\bf $p$-Riordan words.} We define the following set of words that we call {\em $p$-Riordan words} because they will be proved by us in Section~\ref{p-Riordan-sec} to be in 1-to-1 correspondence with $p$-Riordan graphs.  A $p$-Riordan word $w_1w_2\cdots w_n$, $n\geq 1$, is a word over the alphabet $A^{(p)}=\{a_{i,j}:0\leq i,j\leq p-1\}$ such that either it is $a_{0,0}a_{0,0}\cdots a_{0,0}$ or there exist $i$ and $b$, $1\leq i\leq n$, $1\leq b\leq p-1$, such that $w_1\cdots w_i=a_{0,0}a_{0,0}\cdots a_{0,0}a_{b,0}$. By definition, the empty word $\varepsilon$ is a $p$-Riordan word of length $0$. The set of $p$-Riordan words of length $n$ is denoted by $\mathcal{W}_n^{(p)}$. For example, letting $a=a_{0,0}$, $b=a_{1,0}$, $c=a_{0,1}$ and $d=a_{1,1}$, $\mathcal{W}_3^{(2)}=\{aaa,aab,aba,abb,abc,abd,baa,bab,bac,bad,bba,bbb,bbc,bbd,bca,$ \\ $bcb,bcc,bcd,bda,bdb,bdc,bdd\}.$\\

\noindent
{\bf Organization of the paper}. This paper is organized as follows. In Section~\ref{p-Riordan-sec} we present our main result, namely, encoding  $p$-Riordan graphs by $p$-Riordan words. In Section~\ref{pattern-avoiding-sec} we consider encoding Riordan graphs by pattern-avoiding permutations. In Section~\ref{orient-Riord-graphs} we encode oriented Riordan graphs by balanced words over the alphabet $\{0,1,2\}$, and provide, as a bi-product, a proof of a known enumerative result related to the formula (\ref{formula-mod-3}). Finally, in Section~\ref{concl-sec} we make concluding remarks and state open problems for further research.

\section{Encoding $p$-Riordan graphs by $p$-Riordan words}\label{p-Riordan-sec}
In this section, we will present a bijective map $\xi:\mathcal{RG}_{n+1}^{(p)}\rightarrow\mathcal{W}_{n}^{(p)}$ for $n\geq 0$ and discuss its applications in the case of Riordan graphs (the case $p=2$). 

\begin{theorem}\label{thm-xi}
There is bijection between $\mathcal{RG}_{n+1}^{(p)}$ and $\mathcal{W}_{n}$.
\end{theorem}
\begin{proof}
By definition, a $p$-Riordan graph $G=G^{(p)}_n(g,f)$ of order $n$ can be determined by the coefficients $g_{b^*},g_{b^*+1},\ldots,g_{n-2},f_1,f_2,\ldots,f_{n-b^*-2}$ where $b^*$ is the smallest $i\in\{0,1,\ldots,n-2\}$ such that $g_i\neq 0$; if no such $i$ exists then $G$ is an empty graph (a graph with no edges). Thus, $G^{(p)}_n(g,f)=G^{(p)}_n(\tilde{g},\tilde{f})$ where $\tilde{g}$ and $\tilde{f}$ are the following polynomials:
\begin{itemize}
\item $\tilde{g}:=\displaystyle\sum_{i=b^*}^{n-2}g_{i}t^i$ and $\displaystyle\tilde{f}:=\sum_{i=1}^{n-2-b^*}f_{i}t^i$ if $b^*$ is defined;
\item $\tilde{g}=\tilde{f}=0$ if $b^*$ does not exist.
\end{itemize}

If $s^{(p)}_n=|\mathcal{W}_n^{(p)}|$ then $s^{(p)}_1=p$ and 
\begin{eqnarray}\label{spn-formula} s^{(p)}_{n+1}=p^2(s^{(p)}_n-1)+p.\end{eqnarray}
Indeed, $\mathcal{W}_1^{(p)}=\{a_{0,i}:0\leq i\leq p-1\}$ and to generate all elements in $\mathcal{W}_{n+1}^{(p)}$ we can extend every word in $\mathcal{W}_{n}^{(p)}$ different from $a_{0,0}\cdots a_{0,0}$ by any letter in $A^{(p)}$ (which explains the term $p^2(s^{(p)}_n-1)$), and the remaining $p$ elements in $\mathcal{W}_{n+1}^{(p)}$ are given by $\{a_{0,0}\cdots a_{0,0}a_{b,0}: 0\leq b\leq p-1\}$. 

We note that $r^{(p)}_n$ satisfies the same recursion as (\ref{spn-formula}), namely, \begin{eqnarray}\label{rpn-rec-formula} r^{(p)}_{n+1}=p^2(r^{(p)}_n-1)+p\end{eqnarray}
with the initial condition $r^{(p)}_2=p$, which provides an alternative proof of (\ref{mod-p-formula}).  Indeed, the initial condition counts the $p$ graphs on 2 vertices given by $\tilde{g}_0=i$, $0\leq i\leq p-1$. Also, to generate all graphs in $\mathcal{RG}_{n+1}^{(p)}$, we can take any non-empty graph in $\mathcal{RG}_{n}^{(p)}$ given by $\tilde{g}\neq 0$ and $\tilde{f}$, and choose independently $g_{n-1}$ and $f_{n-b^*-1}$ in $\{0,1,\ldots, p-1\}$ ($b^*$ is defined). This explains the term ``$p^2(r^{(p)}_n-1)$'' in (\ref{rpn-rec-formula}). The only uncounted graphs in $\mathcal{RG}_{n+1}^{(p)}$ are those given by $\tilde{g}=it^{n-1}$ where $0\leq i\leq p-1$, which explains the term ``$+p$''.

For the only graph on one vertex, $\xi(G_{1}^{(p)}(0,0))=\varepsilon$, and for $n\geq 2$, $\xi(G_{n}^{(p)}(g,f))$ is defined by \\[-5mm]
$$a_{g_0,f_{n-2}}a_{g_1,f_{n-3}}\cdots a_{g_{b^*-1},f_{n-3-b^*}} a_{g_{b^*},f_0}a_{g_{b^*+1},f_1}a_{g_{b^*+2},f_2}\cdots a_{g_{n-2},f_{n-2-b^*}}=$$
$$a_{0,0}a_{0,0}\cdots a_{0,0} a_{1,0}a_{g_{b^*+1},f_1}a_{g_{b^*+2},f_2}\cdots a_{g_{n-2},f_{n-2-b^*}}$$ where $g_i=[t^i]\tilde{g}$ and $f_j=[t^j]\tilde{f}$. The fact that $\xi$ is well-defined and bijective essentially follows from our proofs of (\ref{spn-formula}) and (\ref{rpn-rec-formula}).
\end{proof}

The bijection $\xi$ given in Theorem~\ref{thm-xi} allows us to encode by words three classes of Riordan graphs (the case $p=2$).  As above, we let $\mathcal{W}_n^{(2)}$ be words over $\{a,b,c,d\}$, where $a=a_{0,0}$, $b=a_{1,0}$, $c=a_{0,1}$ and $d=a_{1,1}$,  Justifications of our encodings are straightforward from the properties of $\xi$.\\

\noindent 
{\bf The class $(1+f,f)$.} In this case, $g=1+f$. Since $[t^0]g=1$ and $[t^i]g=[t^i]f$ for each $i\in[n-2]$, we have that $b^*=0$ and the words in $\mathcal{W}_n^{(2)}$  corresponding to these graphs are those of the form $bx_1x_2\cdots x_{n-1}$ where $x_i\in\{a,d\}$.\\

\noindent
{\bf Proper Riordan graphs.} For a proper Riordan graph $g_0= f_1 = 1$, so $b^*=0$ and the words in $\mathcal{W}_n^{(2)}$ corresponding to these graphs are those beginning either with $bc$ or with $bd$ for $n\geq 2$, and $\mathcal{W}_1^{(2)}=\{b\}$.\\

\noindent 
{\bf Riordan graphs of the Appel type.} For these graphs $f=t$, and thus the words in $\mathcal{W}_n^{(2)}$ corresponding to these graphs are of the form $aa\cdots a$, $aa\cdots ab$ and $aa\cdots abxw$ for $x\in\{c,d\}$ and $w$ a word over the alphabet $\{a,b\}$. \\

\section{Encoding Riordan graphs by pattern-avoiding permutations}\label{pattern-avoiding-sec}

The sequence $(r_n)_{n\ge0}$  counting Riordan graphs and given by (\ref{Riord-rn}) has various combinatorial interpretations recorded in
A047849 in the {\em On-Line Encyclopedia of Integer Sequences} ({\em OEIS})
\cite{oeis}. Two of these interpretations are related to pattern-avoiding permutations: 
\begin{itemize}
\item[(I1)] $r_n$ counts permutations in $S_{2n}(123, 132)$ with two fixed points; the set of such permutations is denoted by $P_{2n}$; see \cite{MR}.
\item[(I2)] $r_n$ counts permutations in $S_n(4321,4123)$, or in $S_n(4321, 3412)$, or in $S_n(4123,3214)$, or in $S_n(4123,2143)$; see \cite{KW}.
\end{itemize}

In this paper, we explain combinatorially a connection between Riordan graphs and (I1). Moreover,  we will show which permutations, under the bijection we will construct, correspond to Riordan graphs of the Appell type, of the Bell type (including the Pascal graph), of the derivative type, as well as of the types given by $(1+f,f)$ and $(g,0)$ (a star graph with a number of isolated vertices), and to proper Riordan graphs. We leave explaining the second bullet point combinatorially as an open problem.

\begin{figure}
\begin{center}
\psscalebox{0.5 0.5} % Change this value to rescale the drawing.
{
\begin{pspicture}(0,-5.2)(10.56,5.2)
\definecolor{colour0}{rgb}{0.8,0.8,0.8}
\psline[linecolor=black, linewidth=0.04](0.38,2.0)(10.38,2.0)
\psline[linecolor=black, linewidth=0.04](0.38,-1.6)(10.38,-1.6)
\psline[linecolor=black, linewidth=0.04](3.58,5.2)(3.58,-4.8)
\psline[linecolor=black, linewidth=0.04](6.78,5.2)(6.78,-4.8)
\psframe[linecolor=black, linewidth=0.04, fillstyle=solid,fillcolor=colour0, dimen=outer](1.98,5.2)(0.38,2.0)
\psframe[linecolor=black, linewidth=0.04, fillstyle=solid,fillcolor=colour0, dimen=outer](3.58,2.0)(1.98,-1.6)
\psframe[linecolor=black, linewidth=0.04, fillstyle=solid,fillcolor=colour0, dimen=outer](6.78,-1.6)(3.58,-3.2)
\psframe[linecolor=black, linewidth=0.04, fillstyle=solid,fillcolor=colour0, dimen=outer](10.38,-3.2)(6.78,-4.8)
\rput[bl](0.02,-4.92){\huge$1$}
\rput{-0.33752647}(0.030576011,0.020646136){\rput[bl](3.52,-5.45){\huge$a$}}
\rput[bl](6.68,-5.45){\huge$b$}
\rput[bl](10.2,-5.45){\huge$2n$}
\rput[bl](-0.1,-1.66){\huge$a$}
\rput[bl](-0.1,1.9){\huge$b$}
\rput[bl](-0.4,4.98){\huge$2n$}
\psframe[linecolor=black, linewidth=0.04, dimen=outer](10.38,5.2)(0.38,-4.8)
\psdots[linecolor=black, dotsize=0.3](3.58,-1.6)
\psdots[linecolor=black, dotsize=0.3](6.78,2.0)
\rput[bl](0.9,3.32){\huge$A$}
\rput[bl](2.5,-0.14){\huge$B$}
\rput[bl](4.9,-2.67){\huge$C$}
\rput[bl](8.2,-4.3){\huge$D$}
\end{pspicture}
}
\end{center}
\caption{The structure of a permutation in $P_{2n}$}\label{fig:P2n}
\end{figure}
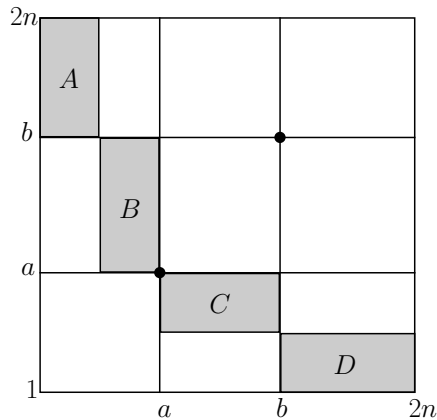

\begin{theorem}\label{thm1}
There is a bijection between $\mathcal{RG}_n$ and $P_{2n}$.
\end{theorem}

\begin{proof} Let $\pi=p_1p_2\cdots p_{2n}$ and $a$ and $b$ be the two fixed points in $\pi$ where $1\le a<b\le 2n$. Since $\pi$ avoids the patterns $123$ and $132$, the remaining $2n-2$ elements can only be placed in  the areas $A$, $B$, $C$ and $D$ in Figure~\ref{fig:P2n}, where we show $\pi$ schematically as a permutation diagram. In addition, the elements in $B$ and $C$ must be in decreasing order, while $A$ (resp., $D$) avoids 123 and 132, and is independent from the rest of $\pi$ in the sense that no occurrence of 123 or 132 can start in $A$ (resp., $D$) and end elsewhere. Since the number of elements in $B$ and $C$ is $b-a-1$, we obtain that the number of elements in $D$, which is the same as the number of elements in $A$, is
\begin{align*}
2n-b=a-1-(b-a-1)\quad\Rightarrow\quad a=n.
\end{align*}
Thus, $B$ and $C$ have the same number of elements $b-n-1$, and $\pi$ satisfies 
\begin{itemize}
\item $p_1p_2\cdots p_{2n-b}$ and $p_{b+1}p_{b+2}\cdots p_{2n}$ avoid the patterns 123 and 132;
\item $p_{2n-b+i}=b-i$ for each $i=1,2,\ldots,2b-2n-1$.
\end{itemize}
The desired map $\phi:\mathcal{RG}_n\rightarrow P_{2n}$,  is defined as follows where $\phi(G_n(g,f))=p_1p_2\cdots p_{2n}$ and  the function $\psi$ is defined in Section~\ref{intro}:
\begin{itemize}
\item $p_n=n$ and $p_{b}=b$ where $b=b^*+n+1$ and $b^*=\min\{i\mid [t^i]g=1\}$  if $g\neq0$ and $b=2n$ if $g=0$;
\item $p_{2n-b+i}=b-i$ for each $i\in[2b-2n-1]$;
\item $\psi(g_{b^*+1}g_{b^*+2}\cdots g_{n-2})=(p_{1}-b)\cdots (p_{2n-b}-b)$ so that to obtain $p_{1}\cdots p_{2n-b}$ we apply $\psi$ to $g_{b^*+1}g_{b^*+2}\cdots g_{n-2}$ and then increase each element in the obtained permutation by $b$ (lengths are proper here since $b=b^*+n+1$);
\item $\psi(f_1f_2\cdots f_{2n-b-1})=p_{b+1}\cdots p_{2n}$.
\end{itemize}
It is not difficult to see that $\phi$ is well-defined and injective. Thus, since $|\mathcal{RG}_n|=|P_{2n}|$, $\phi$ is bijective. For example, $\phi(G_5(t+t^3,t^2))=(10)896547312$ and $\phi(G_5(t,t+t^2))=98(10)6547321$. In particular, in both cases $n=5$, $2n=10$ and $b^*=1$ so that $b=7$. 
\end{proof}

The bijection $\phi$ given in Theorem~\ref{thm1} allows us to describe subclasses of Riordan graphs in terms of pattern-avoiding permutations. In our description we refer to $A$ and $D$ to be the parts of a permutation $\pi=p_1\cdots p_{2n}$ schematically given in Figure~\ref{fig:P2n}. That is, $A=p_1\cdots p_{2n-b}$ and $D=p_{b+1}\cdots p_{2n}$.  Justifications of our descriptions are usually straightforward from the properties of $\phi$, but in some places we still provide various clarifications. In what follows recall that $b=b^*+n+1$.\\

\noindent 
{\bf The class $(1+f,f)$.} In this case, $g=1+f$. Since $[t^0]g=1$ and $[t^i]g=[t^i]f$ for each $i\in[n-2]$, we have that $b^*=0$ and the permutations in $P_{2n}$ corresponding to these graphs have the fixed points $n$ and $n+1$, and red$(A)=D$.\\

\noindent
{\bf Proper Riordan graphs.} For a proper Riordan graph $g_0= f_1 = 1$, so $b^*=0$ and the permutations in $P_{2n}$ corresponding to these graphs have two fixed points $n$ and $n+1$,  and the minimal element in $D$ is in the last place.\\

\noindent 
{\bf Riordan graphs of the Appel type.} For these graphs $f=t$ and thus the permutations in $P_{2n}$ describing them are those for which $D=(2n-b-1)(2n-b-2)\cdots 32(2n-b)1$. \\

\noindent
{\bf Riordan graphs of the Bell type.} Such graphs are given by $(g,tg)$ where either $g=0$ or $g_0=g_1=\cdots=g_{b^*-1}=0$ and $g_{b^*}=1$ for $b^*\geq 0$. In the former case, we deal with the empty graph corresponding to $(2n-1)(2n-2)\cdots 1(2n)$. For the latter case, since  $\tilde{f}=g_0t+g_1t^2+\cdots +g_{n-2-b^*}t^{n-1-b^*}$, we can conclude that permutations in $P_{2n}$ corresponding to the Riordan graphs of the Bell type are those having the $b^*+1$ rightmost elements of $D$ forming the permutation $b^*(b^*-1)\cdots 21(b^*+1)$. We note that the Pascal graph $\left(\frac{1}{1-z},\frac{z}{1-z}\right)$ is the only graph of the Bell type that belongs to the class $(1+f,f)$. For this graph $b^*=0$, $b=n+1$, and the permutation in $P_{2n}$ corresponding to it is $(2n)(2n-1)\cdots (n+2)n(n+1)(n-1)(n-2)\cdots 1$.\\

\noindent
{\bf Riordan graphs of the derivative type.} This class of graphs is of the form $(f',f)$ so $g=f'=f_1+f_3t^2+f_5t^4+\cdots$. Thus,  permutations corresponding to such graphs can be described algorithmically by imposing the following restrictions on $A$: in implementing the bijection $\psi$ every even step must place the current element into the left out of two valid slots. Thus, if this rule is not violated, the obtained permutation corresponds to a Riordan graph of the derivative type.\\

\section{Encoding oriented Riordan graphs by balanced words over $\{0,1,2\}$}\label{orient-Riord-graphs}

The  sequence $\tilde{r}_n$ counting oriented Riordan graphs $\mathcal{RG}^{(3)}_n$ and given by (\ref{formula-mod-3}) is A054879 in \cite{oeis}, one of whose interpretations is the number of words of length $2n$ on alphabet $\{0,1,2\}$ with an even number (possibly zero) of each letter. We call these words {\em balanced words} and denote the respective set by $\mathcal{B}_n$. Also, let $b_n:=|\mathcal{B}_n|$. The words in $\mathcal{B}_n$ are clearly in 1-to-1 correspondence with closed walks of length $2n$ along the edges of the $3$-cube (cube) starting at the origin. Such walks also mentioned in A054879. The correspondence is given by letting a walk take the $i$-th step in direction $x$, $1\leq x\leq 3$ (that is, swap the $i$-th coordinate from 0 to 1, or vice versa) in the case if the $i$-th letter in the word is $x-1$. The formula for $b_n$ (the same as (\ref{formula-mod-3})) is given in A054879 in \cite{oeis}, but its derivation would involve simplifying a more general result in \cite{Reyzin}
$$\frac{1}{2^n}\sum_{j=0}^{n}{n\choose j}(n-2j)^r$$
after placing $n=3$ there. As a bi-product, in this section we obtain an alternative justification of  (\ref{formula-mod-3}) be valid for $b_n$ after explaining combinatorially the recursion (\ref{rec-2-b}).

To build a recursive encoding in question, we will use an alternative, inductive way to prove formula (\ref{formula-mod-3}) via first explaining combinatorially the recursion 
\begin{eqnarray}\label{rec-1-o}\tilde{r}_{n+1}=3\tilde{r}_n+6(\tilde{r}_n-1),\end{eqnarray} 
%or, equivalently, that \begin{eqnarray}\label{rec-2-o}\tilde{r}_{n+1}=3\tilde{r}_n+6(\tilde{r}_n-1).\end{eqnarray} 
which is a particular case of $p=3$ in (\ref{rpn-rec-formula}) proved above. Next, we explain \begin{eqnarray}\label{rec-2-b}b_{n+1}=3b_n+6(b_n-1),\end{eqnarray} which, again, will provide an alternative proof for the formula (\ref{formula-mod-3}) for $b_n$ given in A054879. \\

\noindent
{\bf Explanation of (\ref{rec-1-o}).} Each graph $G_{n+1}(g,f)$ in  $\mathcal{RG}^{(3)}_{n+1}$ either has
\begin{itemize}
\item $f_{n-b^*-1}=0$ (in the case $b^*$ is not defined we can assume $\tilde{f}=0$), 
and there are $3\tilde{r}_n$ ways to select such $G$ , where ``3'' is the number of choices for $g_{n-1}$, or 
\item $f_{n-b^*-1}\in\{1,2\}$ and $g_0g_1\ldots g_{n-2}\neq 00\cdots 0$, so that there are 3 choices for $g_{n-1}$ and in total $6(b_n-1)$ ways to select such a $G$. 
\end{itemize}
This completes the proof of (\ref{rec-1-o}).\\ 

\noindent
{\bf Explanation of (\ref{rec-2-b}).} We note that any balanced word either ends with $xx$ or with $xy$, $x\neq y$, $x,y\in\{0,1,2\}$. Clearly, there are $3b_n$ words in $\mathcal{B}_{n+1}$ ending with 00 or 11 or 22. Next, we prove that there are $6(b_n-1)$ words in $\mathcal{B}_{n+1}$ ending with $xy$, $x\neq y$. To generate any such word we use the function $h_{x,y}$ that 
\begin{itemize}
\item takes any word $w$ in $\mathcal{B}_{n}$ different from $zz\cdots z$ for $z\neq x,y$, then
\item replaces the leftmost occurrence of an element in $\{x,y\}$ in $w$ by the opposite from the same set, and
\item attaches $xy$ to the obtained word.
\end{itemize}
For example, for $x=0$ and $y=1$, $h_{0,1}(22012120)=2211212001$. Note that $h_{x,y}$ is well-defined because its outcome is a balanced word. 

Next note that $h_{x,y}$ is injective. Indeed, different words, say $w_1$ and $w_2$, would either get different endings, and thus will result in different outcomes, or they will have the same ending. In the later case, if such an ending is $xx$ then we clearly have different outcomes. Otherwise, the ending is $xy$ for $x\neq y$. Consider the smallest $i\geq 1$ such that the $i$-th element in $w_1$ is different from the $i$-th element in $w_2$ (such an $i$ exists).  Thus, $w_1=x_1x_2\cdots x_{i-1}a\cdots$ and $w_2=x_1x_2\cdots x_{i-1}b\cdots$ and $a\neq b$. If none of the $a$ and $b$ would be changed by $h_{x,y}$, the outcome words are different. On the other hand if, without loss of generality $a$ will be changed, then $x_1x_2\cdots x_{i-1}=zz\cdots z$ for $z\neq x,y$, and thus either 
\begin{itemize}
\item $b=z$ in which case we obtain $h_{x,y}(w_1)\neq h_{x,y}(w_2)$ as $a\neq z$, or
\item $b$ in $w_2$ will be changed as well, so that $h_{x,y}(w_1)$ and $h_{x,y}(w_2)$ differ in position $i$.
\end{itemize}
So, in either case, $h_{x,y}(w_1)\neq h_{x,y}(w_2)$.

We finally prove surjectivity of $h_{x,y}$ that completes our proof of (\ref{rec-2-b}). Given a word $w_1w_2\cdots w_{2n+2}\in\mathcal{B}_{n+1}$ we replace the leftmost occurrence of a letter in $\{w_{2n+1}w_{2n+2}\}$ in $w_1w_2\cdots w_{2n}$ by the opposite from the same set so that the resulting word belongs to $\mathcal{B}_n$. For example, $h_{x,y}^{-1}(0021210202)=20212102$. \\

\begin{theorem}
There is a bijection between $\mathcal{RG}^{(3)}_{n+1}$ and $\mathcal{B}_n$.
\end{theorem}
\begin{proof}
We describe recursively a bijective map $\eta:\mathcal{RG}^{(3)}_{n+1}\rightarrow\mathcal{B}_n$ based on our proofs of relations (\ref{rec-1-o}) and (\ref{rec-2-b}) with the base cases $\eta(G_2(0,0))=00$, $\eta(G_2(1,0))=11$ and $\eta(G_2(2,0))=22$. 

Suppose that $\eta$ is already defined for $\mathcal{RG}^{(3)}_{n}$, $n\geq 2$, and $\eta(G_n(g,f))=w_1w_2\cdots w_{2n-2}$. Moreover, we make the following assumptions, which work for the base case and will work for $n+1$ as well: 
\begin{itemize}
\item $\eta(G_n(0,0))=00\cdots 0$;
\item  $\eta(G_n(1+t+\cdots + t^{n-2},0))=11\cdots 1$;
\item $\eta(G_n(2+2t+\cdots + 2t^{n-2},0))=22\cdots 2$.
\end{itemize}
Then, adding $g_{n-1}$ and $f_{n-b^*-1}$ into consideration results in a graph $G\in \mathcal{RG}^{(3)}_{n+1}$ obtained from $G_n(g,f)$ by adding one vertex, and we define $\eta(G)$ by the following rules:
\begin{itemize}
\item If $f_{n-b^*-1}=0$ or $\tilde{f}=0$ (in the case $b^*$ is not defined) then $\eta(G)=w_1w_2\cdots w_{2n-2}g_{n-1}g_{n-1}$.
\item If $f_{n-b^*-1}=1$, then we can take any $G_n(g,f)$ except for the one with  $\tilde{g}= 0$ corresponding to $w_1w_2\cdots w_{2n-2}=00\cdots 0$, and consider the following subcases:  
\begin{itemize}
\item if $g_{n-1}=0$ then $\eta(G)=h(w_1w_2\cdots w_{2n-2})01$ if $w_1w_2\cdots w_{2n-2}\neq 22\cdots 2$ and $\eta(G)=h(00\cdots 0)01=100\cdots 01$ otherwise; note that this rule is injective, well-defined for any $G_n(g,f)$ different from an empty graph, and covers all words in $\mathcal{B}_n$ ending with 01.
\item if $g_{n-1}=1$ then $\eta(G)=h(w_1w_2\cdots w_{2n-2})12$, which is well-defined since $w_1w_2\cdots w_{2n-2}\neq 00\cdots 0$, injective, and covers all words in $\mathcal{B}_n$ ending with 12.
\item if $g_{n-1}=2$ then $\eta(G)=h(w_1w_2\cdots w_{2n-2})02$ if $w_1w_2\cdots w_{2n-2}\neq 11\cdots 1$ and $\eta(G)=h(00\cdots 0)02=200\cdots 02$ otherwise; note that this rule is injective, well-defined for any $G_n(g,f)$ different from an empty graph, and covers all words in $\mathcal{B}_n$ ending with 02.\end{itemize}
\item If $f_{n-b^*-1}=2$, then we can take any $G_n(g,f)$ except for the one with  $\tilde{g}= 0$ corresponding to $w_1w_2\cdots w_{2n-2}=00\cdots 0$, and consider the following subcases:  
\begin{itemize}
\item if $g_{n-1}=0$ then $\eta(G)=h(w_1w_2\cdots w_{2n-2})10$ if $w_1w_2\cdots w_{2n-2}\neq 22\cdots 2$ and $\eta(G)=h(00\cdots 0)10=100\cdots 010$ otherwise; note that this rule is injective, well-defined for any $G_n(g,f)$ different from an empty graph, and covers all words in $\mathcal{B}_n$ ending with 10.
\item if $g_{n-1}=1$ then $\eta(G)=h(w_1w_2\cdots w_{2n-2})21$, which is well-defined since $w_1w_2\cdots w_{2n-2}\neq 00\cdots 0$, injective, and covers all words in $\mathcal{B}_n$ ending with 21.
\item if $g_{n-1}=2$ then $\eta(G)=h(w_1w_2\cdots w_{2n-2})20$ if $w_1w_2\cdots w_{2n-2}\neq 11\cdots 1$ and $\eta(G)=h(00\cdots 0)20=200\cdots 020$ otherwise; note that this rule is injective, well-defined for any $G_n(g,f)$ different from an empty graph, and covers all words in $\mathcal{B}_n$ ending with 020.\end{itemize}
\end{itemize}
The fact that $\eta$ is well-defined and bijective follows from our remarks above. For example, $w=\eta(G_5(1+2t^2+t^3,t^2))$ can be recursively calculated from $\eta(G_4(1+2t^2+t^3,t^2))$ and $g_{n-1}=g_4=0$ and $f_{n-b^*-1}=f_{5-1-1}=f_3=0$ so that $w$ will end with 00. In turn, $w'=\eta(G_4(1+2t^2+t^3,t^2))$ can be recursively calculated from $\eta(G_3(1+2t^2+t^3,t^2))$ and $g_3=1$ and $f_{2}=1$, so that $w'$ will end with 12. And so on. One can check that $w=2100221200$.
\end{proof}

\section{Concluding remarks}\label{concl-sec}
This paper provides a convenient encoding of $p$-Riordan graphs in terms of $p$-Riordan words, and explains combinatorially some links between (oriented) Riordan graphs and balanced words (equivalently, certain closed walks in a cube) and pattern-avoiding permutations. 

We leave it as an open question to explain combinatorially that the sequence $(r_n)_{n\ge0}$  counting Riordan graphs and given by (\ref{Riord-rn}) also counts permutations in $S_n(4321,4123)$, or in $S_n(4321, 3412)$, or in $S_n(4123,3214)$, or in $S_n(4123,2143)$ that were enumerated in \cite{KW}; also see A047849 in \cite{oeis}. A natural approach would be to use 2-Riordan words that are in bijection with Riordan graphs to establish a correspondence in question. However, enumeration of pattern-avoiding permutations in \cite{KW} is rather involved (it uses the notion of an {\em active site}) and e.g.\ cannot be translated that easily into recursion (\ref{spn-formula}) for $p=2$.

\section{Acknowledgments} The work of the second author was supported by the National Research Foundation of Korea (NRF) grant funded by the Ministry of Education of Korea (NRF-2019R1I1A1A01044161).

 \end{document}